\documentclass[12pt,reqno]{amsart}
\usepackage{amsmath}
\usepackage[english,  activeacute]{babel}
\usepackage[latin1]{inputenc}
\usepackage{amssymb}
\usepackage{amsthm}
\usepackage{graphics,graphicx}
\usepackage{array}
\usepackage{a4wide}
\allowdisplaybreaks
\usepackage{color, url}
\usepackage{float}
\usepackage{multicol}
\usepackage[shortlabels]{enumitem}
\usepackage[hypertexnames=false,colorlinks=true,allcolors=blue]{hyperref}
\theoremstyle{plain}
\newtheorem{theorem}{Theorem}[section]

\theoremstyle{definition}

\newtheorem{conjecture}[theorem]{Conjecture}
\newtheorem{remark}[theorem]{Remark}

\begin{document}
		\title[Combinatorial proof of a result on generalized  overcubic partitions ]{Combinatorial proof of a result on generalized  overcubic partitions and related conjectures  } 
	
	\author{Suparno Ghoshal and Arijit Jana}
	\address{Department of Computer Science, Ruhr University Bochum, Germany}
	\email{suparno.ghoshal@rub.de}
	\address{Department of Mathematics, National Institute of Technology Silchar, Assam 788010, India}
	\email{jana94arijit@gmail.com}
	
	\thanks{2020 \textit{Mathematics Subject Classification.}  11P83, 05A15, 05A17.\\
		\textit{Keywords and phrases. Overcubic partition, Congruences, }}
	
	\begin{abstract}
		Extending Sellers' result, Das et al. recently proved some congruence results for generalized overcubic partitions using theta functions and posed some related conjectures. In this paper, we provide a combinatorial proof of a result in modulo $4$ of Das et al. and confirm some of their conjectures. 
	\end{abstract}
	\maketitle
\section{Introduction}

An integer \textit{partition} of a positive integer $n$ is defined as a non-increasing sequence of positive integers $\lambda=(\lambda_1, \lambda_2, \ldots, \lambda_k)$ which sums to $n$. Each $\lambda_i$ denotes "parts" of a partition.
We denote by $p(n)$, the number of partitions of $n$. For $q \in \mathbb{C}$  and $|q| < 1$,  throughout the notation $f^k_n$ stands for $f^k_n:=\prod_{j\geq 1}(1-q^{jn})^k$ for all integers $n,k$ with $n>0$.
The generating function of $p(n)$ is written as 
\begin{equation}\label{eq0}
\sum_{n\geq 0}p(n)q^n = \prod\limits_{n=1}^{\infty }\frac{1}{1-q^n}= \frac{1}{f_1}.
\end{equation}
Ramanujan was one of the pioneers who came up with several identities and important congruences related to the partition function $p(n)$ for all integers $n \geq 0$ such as these
\begin{align}
p(5n +4) \equiv 0 \pmod{5},\\
p(7n + 5) \equiv 0 \pmod{7},\\
p(11n + 6) \equiv 0 \pmod{11}.
\end{align}
In fact the following identity \eqref{Ramanujan} was stated as the ``Most Beautiful Identity" according to Hardy and MacMahon
\begin{align}\label{Ramanujan}
\sum_{n \geq 0} p(5n +4)q^{n} \equiv 5 \dfrac{(q^{5}; q^{5})_{\infty}^{5}}{(q; q)_{\infty}^{6}}
\end{align}
where $(a; q)_{\infty} = \prod_{k\geq 0} (1 - aq^k)$. \\
The notion of cubic partitions was introduced by Chan \cite{chan10}  following an identity of Ramanujan's cubic continued fractions \cite{Andrews5}. 
A partition of a non-negative integer $n$ whose even parts can be colored with two colors is usually known as a \textit{cubic partition} of weight $n$. We abide by the known literature of cubic partitions and denote it as $a_{2}(n)$ and set $a_{2}(0) = 1$. For any counting problem it often makes sense to see with examples as to how the counting works. We provide the case for $n = 4$ and we make use of the colors red and green to color the parts, further we assume that the odd parts of a partition can only be colored with the color red, then the number of cubic partitions amounts to $9$ in the following manner
\[
4_{\text{R}}, 4_{\text{G}}, 3_{\text{R}} + 1_{\text{R}}, 2_{\text{R}} + 2_{\text{R}}, 2_{\text{G}} + 2_{\text{G}}, 2_{\text{R}} + 2_{\text{G}}, 2_{\text{R}} + 1_{\text{R}} + 1_{\text{R}}, 2_{\text{G}} + 1_{\text{R}} + 1_{\text{R}}, 1_{\text{R}} + 1_{\text{R}} + 1_{\text{R}} + 1_{\text{R}}.
\]
The generating function of $a_{2}(n)$ is pretty evident from the way it's defined combinatorially
\begin{align}
\sum_{n \geq 0} a_{2}(n)q^{n} = \prod_{j \geq 1}\dfrac{1}{(1 - q^{j})(1 - q^{2j})} = \dfrac{1}{f_{1}f_{2}}.
\end{align}
Chan further was able to come up with an identity similar to the one given by Ramanujan \eqref{Ramanujan}
\begin{align}
\sum_{n \geq 0} a_{2}(3n +2)q^{n} = 3 \dfrac{(q^{3}; q^{3})_{\infty}^{3}(q^{6}; q^{6})_{\infty}^{3}}{(q; q)_{\infty}^{4}(q^{2}; q^{2})_{\infty}^{4}}.
\end{align}
This straightaway gives the congruence $a_{2}(3n +2) \equiv 0 \pmod{3}$. In the literature there has been quite a lot of work that has been carried out since then. Curious readers can go through these works by mathematicians like Chan \cite{chan10b}, Xiong \cite{xiong11}, Sellers \cite{sellers14}, Gireesh \cite{gireesh17}, Yao \cite{yao19}, Hirschhorn \cite{h20}, Merca \cite{merca22}, Baruah \& Das \cite{baruah22}, and Buragohain \& Saikia \cite{buragohain24}, who all studied the arithmetic properties of cubic partitions.\\
Generalised cubic partitions and generalised cubic overpartitions have been the subject of recent research by Amdeberhan et al. \cite{Amde25}, Das et al. \cite{Das25}. A generalized cubic partition of weight $n$ is a partition of $n\geq 1$ wherein each even part may appear in $c\geq 1$ different colors. In \cite{Amde25},  they denote  the number of such generalized cubic partitions of $n$ by $a_c(n)$, with $a_c(0):=1$ for all $c\geq 1$.
The generating function for $a_c(n)$ is given by
\begin{equation}\label{p3eq1}
\sum_{n\geq 0}a_c(n)q^n=\frac{1}{f_1f_2^{c-1}}.
\end{equation}
Setting $c=1$ yields the ordinary integer partitions, and setting $c=2$ yields the \textit{cubic partitions} of $n$.\\
We will first discuss briefly about overpartitions before getting into the details of generalized overcubic partitions. Corteel and Lovejoy introduced the generating function for counting the number of overpartitions of a positive integer $n$ in their work \cite{corteel4}
\begin{align}
    \sum_{n \geq 0} \overline{p}(n)q^{n} = \dfrac{f_{2}}{f_{1}^{2}}.
\end{align}
Combinatorially an overpartition of a positive integer $n$ means a partition of $n$ where the first occurrence of each part size may be overlined. For example, $\overline{p}(3)=8$, counts the following overpartitions
$$(3),\, (\overline{3}),\, (2,1),\, (\overline{2},1),\ (2,\overline{1}),\, (\overline{2},\overline{1}),\, (1,1,1),\, (\overline{1},1,1).$$
Over the years, people have carried out extensive research in the area of overpartitions, some of which are cited here \cite{ h5, kim1, kim2, xia1, xia2}.\\
Amdeberhan et al. \cite{Amde25} also considered \textit{generalized overcubic partitions}. These are defined in the same way as generalized cubic partitions and denoted by $\bar a_c(n)$. The generating function for $\bar a_c(n)$ is given by
\begin{equation} \label{p3eq2}
\sum_{n\geq 0}\bar a_c(n)q^n=\frac{f_4^{c-1}}{f_1^2f_2^{2c-3}}.
\end{equation}

New congruences related to the generalized overcubic partitions were provided by them \cite{Amde25} along with their proofs which involved modular forms and elementary theta functions. Guadalupe \cite{gadulpade25} followed up by giving more congruences for the $a_c(n)$ function. In a more recent work by Das et al. \cite{Das25}, they extended the congruence list by providing more congruences for both $a_c(n)$ and $\bar a_c(n)$.  \\

\subsection{Overview of this Paper:} The main goal of this paper is to offer combinatorial proof for the congruences associated with the generalized overcubic partitions $\bar{a}_{c}(n)$ modulo $4$ and also give proofs of some of the conjectures made by Das et al. in \cite{Das25}. We begin by providing the combinatorial proof of the following congruences in Section \ref{sec2}.

\begin{theorem}\cite[Theorem 1.4]{Das25}\label{p3thm1}
	For all $n \geq 1$, we have
	$$
	\bar a_c(n) \equiv
	\begin{cases} 
	2 \pmod{4} & \text{if } n = k^2,  \, k \in \mathbb{Z}, \\ 
	2(c+1) \pmod{4} & \text{if } n = 2k^2,  \, k \in \mathbb{Z}, \\ 
	0 \pmod 4 & \text{otherwise}.
	\end{cases}
	$$
\end{theorem}

\begin{remark}
	 Theorem \ref{p3thm1}, which generalized a result of Sellers \cite[Theorem 2.5]{sellers14}, was proved by Das et. al \cite{Das25} by using theta function identities. In this paper we provide a combinatorial argument for proving these generalized set of congruences.
\end{remark}
Section \ref{sec3} consists of elementary proofs for the following conjectures proposed in the paper \cite{Das25}. It is important to note that we needed to provide only elementary proofs of those congruences modulo $3$ as the rest follow immediately from the general result of Section \ref{sec2}.
\begin{conjecture}\cite[Conjecture 7.3]{Das25}\label{p3conj1}
	For all $n\geq 0$ and $i\geq 1$, we have
	\begin{align*}
	\bar a_{3i+2}(9n+2)&\equiv 0 \pmod {6},\\
	\bar a_{3i+2}(9n+5)&\equiv 0 \pmod{6},\\
	\bar a_{3i+2}(9n+8)&\equiv 0 \pmod{6},\\
    \bar a_{9i+5}(9n+3)&\equiv 0 \pmod{12},\\
	\bar a_{9i+8}(9n+3)&\equiv 0 \pmod{12}.\\
	\end{align*}
\end{conjecture}

\begin{theorem}\label{p3thm2}
		For all $n\geq 0$ and $i\geq 1$, we have
	\begin{align}
	\bar a_{3i+2}(3n+2)&\equiv 0 \pmod {6},\label{p3cong1}\\
	\bar a_{9i+5}(9n+3)&\equiv 0 \pmod{12},\label{p3cong2}\\
	\bar a_{9i+8}(9n+3)&\equiv 0 \pmod{12}.\label{p3cong3}
	\end{align}
\end{theorem}
\begin{remark}
   The first three congruences of Conjecture \ref{p3conj1} immediate from \eqref{p3cong1}.
\end{remark}

\section{Combinatorial Proof of Theorem \ref{p3thm1}}\label{sec2}
\begin{proof}[Proof of Theorem \ref{p3thm1}]

	Let's decompose the function $\overline{a}_{c}(n)$ in the following way:
	\begin{itemize}
		\item a set of partitions containing only single unique part, what we mean by that is $\lambda = (\lambda_{1}^{k})$ such that $k\cdot \lambda_{1} = n$. Let's denote this set as $P_{1}(n)$.
		\item remaining are those partitions which have two or more distinct parts allowing possible multiplicities. We denote this class as $P_{\geq 2}(n)$. Mathematically, if we take a partition $\lambda = (\lambda_1^{k_{1}}, \lambda_2^{k_{2}}, \cdots, \lambda_r^{k_{r}})$, where $k_{i}, ( \text{for}~ i \in \{1,\dots, r\} )$ are the multiplicities of each unique part $\lambda_{i}$, then $\lambda \in P_{\geq 2}(n)$ iff $r \geq 2$.  
	\end{itemize}
	We can now write $\overline{a}_{c}(n)$ in terms of the sum of the cardinality of the two sets $P_{1}(n)$ and $P_{\geq 2}(n)$
	\begin{align}
	\overline{a}_{c}(n) = |P_{1}(n)| + |P_{\geq 2}(n)|
	\end{align}

    Now we try and find the cardinality of the set $P_{\geq 2}(n)$. We claim that we can express the cardinality of the set $P_{\geq 2}(n)$ in the following way

\begin{equation}\label{eq:overpartition}
    \mid P_{\geq 2}(n)\mid = \sum_{r = 2}^{\max f(\lambda)} 2^{r} \cdot \chi(r),
\end{equation}

where \( \chi(r) \) denotes the number of partitions of \( n \) into exactly \( r \) distinct parts, and \( f(\lambda) \) is a function that maps a partition \( \lambda \) of \( n \) to the number of distinct parts in \( \lambda \).

To understand why equation~\eqref{eq:overpartition} holds, consider a partition with \( r \) distinct parts. For each such partition, the number of corresponding overpartitions is given by the number of ways to choose which of the \( r \) parts are overlined. This is:

\begin{equation}
    \binom{r}{0} + \binom{r}{1} + \binom{r}{2} + \cdots + \binom{r}{r} = 2^{r},
\end{equation}

where \( \binom{r}{i} \) represents the number of ways to overline exactly \( i \) out of the \( r \) distinct parts.

Hence from \eqref{eq:overpartition}, we can easily see that $\mid P_{\geq 2}(n) \mid \equiv 0 \pmod{4}$. We are now only left with checking the cardinality of the set $P_{1}(n)$.

	The set $P_{1}(n)$ clearly consists only of partitions whose parts are divisors of $n$, and its cardinality should reflect this constraint. Hence, we can express $|P_{1}(n)|$ by noting that a partition consisting of a single unique part contributes two overpartitions, as follows:
	\begin{align}\label{comb51}
	|P_{1}(n)| = 2 \times (\mid \kappa_{1}(n) \mid + \mid \kappa_{2}(n) \mid)
	\end{align}
	where the $\kappa_{1}(n)$ denotes the set partitions consisting of single unique odd part and $\kappa_{2}(n)$ is the set of partitions consisting of single unique even part. Since we are considering overcubic partitions we know that the odd parts can only be colored using a single color. On the other hand the even parts can be colored using $c$ different colors. 
    So we further divide the set $\kappa_{\text{2}}(n)$ into two sets $\kappa_{21}(n)$ and $\kappa_{22}(n)$
    \begin{align}\label{eq:ov2}
        |P_{1}(n)| = 2 \times (\mid \kappa_{1}(n) \mid + \mid \kappa_{21}(n)\mid ) + \mid \kappa_{22}(n) \mid
    \end{align}
    Here $\kappa_{21}(n)$ denotes the set of partitions consisting of monochromatic single even part, while $\kappa_{22}(n)$ is the set of partitions consisting of polychromatic single even parts. In case we have partitions consisting of polychromatic single even parts, one can treat them as partitions having more than one unique part. Then with the same logic as in \eqref{eq:overpartition} we can write 
    \begin{equation}\label{eq:over1}
    \mid \kappa_{22}(n)\mid = \sum_{e = 2}^{c} 2^{e} \cdot \tilde{\chi}(e),
\end{equation}
Here $\tilde{\chi}(e)$ counts the number of partitions consisting of single  even parts colored with $e \geq 2$ different colors. From \eqref{eq:over1} we get that $\mid\kappa_{22}\mid \equiv 0 \pmod{4}$. This leaves us with checking only the cardinality of the set $\kappa_{21}(n)$. We claim the following for the cardinality of $\kappa_{21}(n)$
\begin{align}\label{eq:over3}
    \mid \kappa_{21}(n) \mid = \tau_{\text{even}}(n) \cdot c
\end{align}
where $\tau_{\text{even}}(n)$ counts the number of even divisors of $n$.
The above equation \eqref{eq:over3} holds as we are doing monochromatic coloring of partitions consisting of single even part.\\
Now we can re-write equation \eqref{eq:ov2} as following
\begin{align}\label{eq:ov3b}
        |P_{1}(n)| \equiv 2 \times (\mid \kappa_{1}(n) \mid + \tau_{even}(n) \cdot c)
    \end{align}
Again we know that $\mid\kappa_{1}(n)\mid$ is nothing but the number of odd divisors of $n$. Hence we can write \eqref{eq:ov3b} as
\begin{align}\label{eq:ov3}
        |P_{1}(n)| \equiv 2 \times (\tau_{odd}(n) + \tau_{even}(n) \cdot c)
    \end{align}
    $\tau_{odd}(n)$ is the number of odd divisors of $n$. From here on we check the following three cases of Theorem \ref{p3thm1}
    
        { \bf Case 1 ( If $n = k^2$):}  Let's write the prime factorization of $n$
	\begin{align}\label{primefactor}
	n = \prod_{j = 1}^{s} p_{j}^{\alpha_{j}}
	\end{align}
    Since $n$ is a square, $\forall j \in \{1, \ldots, s\}\text{, } \alpha_{j} \equiv 0 \pmod{2}$. Now if we have $n$ as an odd square then the proof follows as $\tau_{even}(n) = 0$, and $\tau_{odd}(n) = \prod_{j = 1}^{s} (\alpha_{j} + 1) \equiv 1 \pmod{2}$. For the other case when $n$ is an even square, w.l.o.g assume that $p_{1} = 2$. Then we have $\tau_{odd}(n) = \prod_{j = 2}^{s} (\alpha_{j} + 1) \equiv 1 \pmod{2}$ and $\tau_{even}(n) = \alpha_{1}\cdot \prod_{j = 2}^{s} (\alpha_{j} + 1) \equiv 0 \pmod{2}$. Combining the above two sub-cases and equation \eqref{eq:ov3}, we can see that $\mid P_{1}(n) \mid \equiv 2 \pmod{4}$, whenever $n$ is a square. \\
    { \bf Case 2 ( If $n = 2 \cdot k^2$):} Now, we consider $n = 2 \cdot k^2 \text{ where } k \in \mathbb{N}$. This case is equivalent to the case that $n$ is of the form $2^{\alpha_{1} + 1}\cdot \tilde{k}^{2}, \text{ where } \alpha_{1}$ is an even non-negative integer and $\tilde{k}$ is a non-negative odd integer. It can be checked easily that for this case $\tau_{even}(n) = (\alpha_{1} + 1)\prod_{j = 2}^{s}(\alpha_{j} + 1) \equiv 1 \pmod{2}$, assuming $\tilde{k}^{2} = \prod_{j = 2}^{s} p_{j}^{\alpha_{j}}$. This will imply that $\tau_{even} \cdot c \equiv c \pmod{2}$. We can also check that $\tau_{odd}(n) \equiv 1 \pmod{2}$. So we can write equation \eqref{eq:ov3} as follows:
    \begin{align}\label{comb91}
	\mid P_{1}(n) \mid \equiv 2(1 +  c) \pmod{4}
	\end{align}
   { \bf  Case 3 (Otherwise):} This case is equivalent to saying $n$ is of the following form. 
    \begin{align}\label{primefactor1}
	n = 2^{\alpha_{1}}\prod_{j = 2}^{s} p_{j}^{\alpha_{j}}
	\end{align}
    where for some $ j \in \{2, \ldots, s\}, \text{ such that } \alpha_{j} \equiv 1 \pmod{2}$. 
    Again we first check that $\tau_{odd}(n) = \prod_{j = 2}^{s} (1 + \alpha_{j}) \equiv 0 \pmod{2}$ and $\tau_{even}(n) = \alpha_{1} \prod_{j = 2}^{s} (1 + \alpha_{j}) \equiv 0 \pmod{2}$. Using the above stated facts and equation \eqref{eq:ov3} we get
    \begin{align}
        \mid P_{1}(n) \mid \equiv 0 \pmod{4}.
    \end{align}
    The proof is now complete.
    \end{proof}

\section{Proof of Theorem \ref{p3thm2}} \label{sec3}

In this section, first we recall some results on Ramanujan's theta function to prove Theorem \ref{p3thm2}.  Following  \cite[Equation 1.2.1]{Spirit}, one can define the Ramanujan's theta function $f(a, b)$ as 
\begin{align*}
f(a,b) := \sum_{k=-\infty}^{\infty} a^{\frac{k(k+1)}{2}}b^{\frac{k(k-1)}{2}}, \quad |ab| < 1.
\end{align*}
One can see from \cite[ p.35, p. 350 and p. 49]{bcb391} that the following particular results hold

	\begin{align}
&	\psi(q) := f(q,q^3) =\dfrac{f_2^2}{f_1},\label{psi}\\
& \psi(-q)=\frac{f_1f_4}{f_2}.\label{psi(-q)}\\
&	\psi(q)=f(q^3,q^6)+q\psi(q^9).\label{3-dissec-psi}\\
&	f(-q,q^2)=\frac{\varphi(q^3)}{\chi(q)}, \label{f(-q,q^2)}
	\end{align}
	where
\begin{align}
&\chi(q):=\dfrac{f_2^2}{f_1f_4}, \label{chiq}\\
&\varphi(q) := f(q,q) =\dfrac{f_2^5}{f_1^2f_4^2}.\label{varphi}
\end{align}
The following identity  
\begin{align}
\dfrac{f_2}{f_1f_4} &= \dfrac{f_{18}^9}{f_3^2f_9^3f_{12}^2f_{36}^3}+q\dfrac{f_6^2f_{18}^3}{f_3^3f_{12}^3}+q^2\dfrac{f_6^4f_9^3f_{36}^3}{f_3^4f_{12}^4f_{18}^3}.\label{eq32}
\end{align}
was developed by Toh \cite[Lemma 2.1, (2.1c)]{toh} as a $3$-dissection of the generating function for the number of partitions of $n$ with distinct odd parts.

\begin{proof}[Proof of Equation \eqref{p3cong1}]
	Putting $c=3i+2$ in \eqref{p3eq2}, we have
	\begin{align}\label{p3eq3}
\sum_{n\geq 0}\bar a_{3i+2}(n)q^n&=\frac{f_{4}^{3i+1}}{f_1^2f_{2}^{6i+1}} \notag \\
	&\equiv \frac{f_{12}^{i+1}}{f_{6}^{2i+1}} \bigg(\frac{f_2}{f_1 f_4}\bigg)^2 \pmod3.
	\end{align}
	
	Now, using  \eqref{eq32} in \eqref{p3eq3}, we see that   the terms of the form  $q^{3n+2}$ in \eqref{p3eq3} are following:
	\begin{align*}
& \frac{f_{12}^{i+1}}{f_{6}^{2i+1}}\bigg(q^2\frac{f_{6}^{4} f_{18}^{6}}{f_{3}^{6} f_{12}^{6}}+ 2 q^2\frac{f_{6}^{4} f_{18}^{6}}{f_{3}^{6} f_{12}^{6}}\bigg)
	\equiv 0 \pmod3 .
	\end{align*}
	As a result, 	\begin{align}\label{p3eq7}
	\sum_{n\geq 0}\bar a_{3i+2}(3n+2)q^n&=0 \pmod3.
	\end{align}
Again, Theorem \ref{p3thm1} implies 
\begin{align} \label{p3eq34}
    \overline{a}_{3i + 2}(3n + 2) \equiv 0 \pmod{2}.
\end{align}
Combining \eqref{p3eq34} and \eqref{p3eq7}, we complete the proof.	
\end{proof}

\begin{proof}[Proof of Equation \eqref{p3cong2}]
	Putting $c=9i+5$ in \eqref{p3eq2}, we have
	\begin{align}\label{p3eq4}
	\sum_{n\geq 0}\bar a_{9i+5}(n)q^n&=\frac{f_{4}^{9i+4}}{f_1^2f_{2}^{18i+7}} \notag \\
	&\equiv \frac{f_{12}^{3i+1}}{f_{6}^{6i+2} f_3} \frac{f_1 f_4}{f_2} \pmod3\notag \\
	 	&= \frac{f_{12}^{3i+1}}{f_{6}^{6i+2} f_3}\psi(-q)\notag \\
	 		&= \frac{f_{12}^{3i+1}}{f_{6}^{6i+2} f_3} \bigg( f(-q^3,q^6)-q\psi(-q^9)\bigg).
	\end{align}
	
 Extracting the terms $q^{3n}$ from \eqref{p3eq4} and then replacing $q^3$ by $q$ we get 
 	\begin{align}\label{p3eq5}
 \sum_{n\geq 0}\bar a_{9i+5}(3n)q^n
 & =\frac{f_{4}^{3i+1}}{f_{2}^{6i+2} f_1}  f(-q,q^2)  \notag \\
  &\equiv \frac{f_{4}^{3i+1}}{f_{2}^{6i+2} f_1} \varphi(q^3) \frac{f_1 f_4}{f_2^{2}} \notag \\
   &\equiv \frac{f_{12}^{i}}{f_{6}^{2i+1} } \varphi(q^3) \frac{ f_4^{2}}{f_2} \notag \\
    &\equiv \frac{f_{12}^{i}}{f_{6}^{2i+1} } \varphi(q^3) \psi(q^2) \notag \\
    &\equiv \frac{f_{12}^{i}}{f_{6}^{2i+1} } \varphi(q^3) \bigg(  f(q^6,q^{12})+q^2\psi(q^{18})\bigg) \pmod3.
 \end{align}
 	There are no terms of the  form $q^{3n+1}$ in \eqref{p3eq5}.  As a result 
 		\begin{align}\label{p3eq6}
 	\sum_{n\geq 0}\bar a_{9i+5}(9n+3)q^n \equiv 0	 \pmod3.
 	\end{align}
 	Since $9n + 3$ can never be a  square or twice a square for any non-negative integer $n$ from Theorem \ref{p3thm1} we can say $\sum_{n\geq 0}\bar a_{9i+5}(9n+3)q^n \equiv 0	 \pmod4$ and further by using \eqref{p3eq6} we complete the proof.
\end{proof}
\begin{proof}[Proof of Equation \eqref{p3cong3}]
	For $c=9i+8$, we have
	\begin{align}\label{p3eq9}
	\sum_{n\geq 0}\bar a_{9i+8}(n)q^n&=\frac{f_{4}^{9i+7}}{f_1^2f_{2}^{18i+13}} 
	\equiv \frac{f_{12}^{3i+2}}{f_{6}^{6i+4} f_3} \frac{f_1 f_4}{f_2} \pmod3.\notag \\
	\end{align}
Following an argument similar to the proof of \eqref{p3cong2}, one can easily complete the proof.
\end{proof}

\section{Acknowledgement} 
We sincerely thank Prof. Ken Ono for his valuable comments on the initial draft of this paper. Suparno Ghoshal was supported by the Deutsche Forschungsgemeinschaft (DFG, German Research Foundation) under Germany's Excellence Strategy - EXC 2092 CASA - 390781972.
\section{Conflict of interest}
On behalf of all authors, the corresponding author states that there is no conflict of interest.

\end{document}